\documentclass[12pt,draft]{amsart}
\usepackage{amssymb}
\usepackage{enumitem}
\usepackage{graphicx}
\usepackage{amssymb}
\usepackage{enumitem}
\usepackage{graphicx}
\usepackage{enumerate, latexsym, amsmath, amsfonts, amssymb, amsthm, graphicx, color, amsmath,
amscd,color}
\usepackage[all]{xy}
\def\pmod #1{\ ({\rm{mod}}\ #1)}
\def\Z{\Bbb Z}

\def\Q{\Bbb Q}

\def\bg{\bigg}
\def\({\bg(}
\def\){\bg)}

\def\sgn{{\rm sgn}}

\def\sgn{{\rm sgn}}

\def\char{{\rm char}}
\def\inv{{\rm inv}}

\def\ve{\varepsilon}

\def\Ack{\medskip\noindent {\bf Acknowledgments}}
\theoremstyle{plain}
\newtheorem{theorem}{Theorem}

\newtheorem{lemma}{Lemma}
\newtheorem{corollary}{Corollary}
\newtheorem{conjecture}{Conjecture}
\theoremstyle{definition}

\theoremstyle{remark}
\newtheorem{remark}{Remark}

\makeatletter
\@namedef{subjclassname@2010}{%
  \textup{2010} Mathematics Subject Classification}
\makeatother
 \vspace{4mm}

\begin{document}
 \baselineskip=17pt
\hbox{} {}
\medskip
\title[H.-L. Wu, Y.-F. She, L.-Y. Wang]
{Cyclotomic matrices and hypergeometric functions over finite fields}
\date{}
\author[H.-L. Wu, Y.-F. She, L.-Y. Wang]{Hai-Liang Wu, Yue-Feng She and Li-Yuan Wang}

\thanks{
\newline \indent  The first author was supported by the National Natural Science Foundation of China (Grant No. 12101321 and Grant No. 11971222) and the Natural Science Foundation of the Higher Education Institutions of Jiangsu Province (21KJB110002). The third author was supported by the Natural Science Foundation of the Higher Education Institutions of Jiangsu Province (21KJB110001).}

\address {(Hai-Liang Wu) School of Science, Nanjing University of Posts and Telecommunications, Nanjing 210023, People's Republic of China}
\email{\tt whl.math@smail.nju.edu.cn}

\address {(Yue-Feng She) Department of Mathematics, Nanjing
University, Nanjing 210093, People's Republic of China}
\email{{\tt she.math@smail.nju.edu.cn}}

\address {(Li-Yuan Wang) School of Physical and Mathematical Sciences, Nanjing Tech University, Nanjing 211816, People's Republic of China}
\email{\tt wly@smail.nju.edu.cn}

\begin{abstract}
With the help of hypergeometric functions over finite fields, we study some arithmetic properties of cyclotomic matrices involving characters and binary quadratic forms over finite fields. Also, we confirm some related conjectures posed by Zhi-Wei Sun.
\end{abstract}
\maketitle
\section{Introduction}

Determinants of special matrices have significant applications in many branches of mathematics. Readers may refer to Krattenthaler's survey paper \cite{K2} for recent progress and advanced techniques on this topic. In this paper, we mainly focus on the arithmetic properties of some determinants of cyclotomic matrices concerning characters of finite fields.
\subsection{Notations and Some History on This Topic}
Let $n$ be a positive integer and let $R$ be a commutative ring. Throughout this paper, for each $n\times n$ matrix $M=[a_{ij}]_{1\le i,j\le n}$ with $a_{ij}\in R$, we use $\det M$ or $|M|$ to denote the determinant of $M$. Also, we use the symbol $\mathbb{F}_q$ to denote the finite field of $q$ elements and $\char(\mathbb{F}_q)$ to denote the characteristic of $\mathbb{F}_q$. When $2\nmid q$ we let $\phi$ denote the unique quadratic multiplicative character of $\mathbb{F}_q$, which is the map $\mathbb{F}_q\rightarrow\mathbb{C}$ that sends $0$ to $0$, each non-zero square to $1$, and each non-square to $-1$. 

Let $p$ be an odd prime. Lehmer \cite{Lehmer} investigated the matrix
$$
L_p(a,b,c,d):=\left[a+b\phi(i)+c\phi(j)+d\phi(ij)\right]_{1\le i,j\le p-1},
$$
where $a,b,c,d$ are complex numbers. In \cite[Theorem 1]{Lehmer} Lehmer showed that for each positive integer $k$ we have
$$
L_p(a,b,c,d)^{k}=(p-1)^{k-1}\left[a_k+b_k\phi(i)+c_k\phi(j)+d_k\phi(ij)\right]_{1\le i,j\le p-1},
$$
where $a_k,b_k,c_k,d_k$ are complex numbers defined by
$$
\left[
\begin{array}{cccccc}
 a_k & b_k\\
 c_k & d_k
\end{array}
\right]=
\left[
\begin{array}{cccccc}
 a & b\\
 c & d
\end{array}
\right]^k.
$$
Moreover, Lehmer \cite[Theorem 2]{Lehmer} also determined all the roots of characteristic polynomial of $L_p(a,b,c,d)$. In the same paper, Lehmer also studied the matrix
$$
M_p(\alpha,\beta):=\left[\alpha+\phi(\beta+i+j)\right]_{1\le i,j\le p-1},
$$
where $\alpha$ is a complex number and $\beta$ is an integer. In \cite[Theorem 3]{Lehmer} Lehmer explicitly computed the characteristic polynomial of $M_p(\alpha,\beta)$.

Along this line, Carlitz \cite{carlitz} generalized Lehmer's results to the multiplicative characters modulo an arbitrary positive integer $m$ (see \cite[Theorem 2]{carlitz}). Furthermore, let $p$ be an odd prime and let $\chi$ be a multiplicative character modulo $p$. For any complex number $\gamma$, Carlitz \cite[Theorem 4]{carlitz} determined the characteristic polynomial of the matrix $C_p(\chi,\gamma)$,
where
$$
C_p(\chi,\gamma):=\left[\gamma+\chi(j-i)\right]_{1\le i,j\le p-1}.
$$

In recent years, by using sophisticated matrix decompositions, Chapman \cite{chapman,evil} and Vsemirnov \cite{M1,M2} studied many variants of the above results. For example, when $p\equiv 1\pmod 4$ is a prime, let $\varepsilon_p>1$ and $h(p)$ denote the fundamental unit and class number of the real quadratic field $\mathbb{Q}(\sqrt{p})$ respectively. Write
$$
\ve_p^{(2-(-1)^{(p-1)/4})h(p)}=a_p'+b_p'\sqrt{p}
$$
with $a_p',b_p'\in\Q$. Then Vsemirnov \cite{M1,M2} confirmed the ``evil determinants" conjecture posed by Chapman \cite{evil}, which states that
$$
\det\left[\phi(j-i)\right]_{1\le i,j\le \frac{p+1}{2}}=
\begin{cases}
 -a_p'   &\mbox{if}\ p \equiv 1 \pmod 4,\\
1 &\mbox{if} \ p\equiv 3 \pmod 4.
\end{cases}
$$

In 2019, Sun \cite{ffadeterminant} studied some matrices involving $\phi$ and quadratic forms over $\mathbb{F}_p$. For instance, Sun \cite[Theorem 1.2(ii)]{ffadeterminant} showed that
$$
S(d,p):=\det[\phi(i^2+dj^2)]_{1\le i,j\le (p-1)/2}=0
$$
whenever $d$ is a quadratic non-residue modulo the odd prime $p$. Also, Sun \cite{ffadeterminant} posed many conjectures involving the determinants of the form
$$
\det\left[\phi(f(i,j))\right]_{1\le i,j\le p-1},
$$
where $f(x,y)$ is a quadratic form over $\mathbb{F}_p$. Readers may refer to \cite{Dimitry, Wu} for the recent progress on this topic.
\subsection{Hypergeometric Functions over Finite Fields}
In 1987, Greene \cite{Greene} initiated the investigations of hypergeometric functions over finite fields. Greene showed that these functions have many properties which are analogous to the ordinary hypergeometric functions. Let $q$ be an odd prime power and let $\chi$ be a multiplicative character of $\mathbb{F}_q$. We also define $\chi(0)=0$.

Suppose $A,B$ are two multiplicative characters of $\mathbb{F}_q$. Let $\varepsilon$ be the trivial character and $\overline{A}$ denote the inverse of $A$, i.e., $A \overline{A}=\varepsilon$.  As an appropriate analog of the binomial coefficients, Greene \cite[Definition 2.4]{Greene} defined
$$
\binom{A}{B}:=\frac{B(-1)}{q}J(A,\overline{B}),
$$
where $J(A,\overline{B})$ is the Jacobi sum defined by 
$$J(A,\overline{B})=\sum_{x\in\mathbb{F}_q}A(x)\overline{B}(1-x).$$
With the above notations, Greene \cite[Definition 3.10]{Greene} gave the following definition. For multiplicative characters $A_0,A_1,\cdots,A_n, B_1, B_2,\cdots, B_n$ of $\mathbb{F}_q$ and $x\in\mathbb{F}_q$, the Gaussian hypergeometric function with respect to the above characters is defined by
$$
_{n+1}F_n\left(
\begin{array}{cccccc}
 A_0&A_1&\cdots&A_n\\
  &B_1&\cdots&B_n
\end{array}\bigg|x
\right)
=
\frac{q}{q-1}\sum_{\chi}\binom{A_0\chi}{\chi}\binom{A_1\chi}{B_1\chi}\cdots\binom{A_n\chi}{B_n\chi}\chi(x),
$$
where the summation runs over all multiplicative characters of $\mathbb{F}_q$. In particular, for multiplicative characters $A,B,C$ we have (see \cite[Theorem 3.6]{Greene})
\begin{equation}\label{Eq. 2F1}
_{2}F_1\left(
\begin{array}{cccccc}
 A & B\\
  & C
\end{array}\bigg|x
\right)
=
\varepsilon(x)\frac{BC(-1)}{q}\sum_{y\in\mathbb{F}_q}B(y)\overline{B}C(1-y)\overline{A}(1-xy).
\end{equation}
 The values of hypergeometric functions $ _{2}F_1$ were extensively investigated and have close relations with the number of rational points on algebraic curves over finite fields. For example, readers may refer to \cite{Bar1, Fuselier, Ono} for details.

\subsection{Main Results}
Let $q=2n+1$ be an odd prime power and let $\chi$ be a generator of the cyclic group of all multiplicative characters of $\mathbb{F}_q$. Let
$\mathbb{F}_q^{\times}:=\mathbb{F}_q\setminus\{0\}$ and let
$$
\mathbb{F}_q^{\times2}:=\{x^2:\ x\in\mathbb{F}_q^{\times}\}=\{\alpha_1,\alpha_2,\cdots,\alpha_n\}.
$$
For an arbitrary $d\in\mathbb{F}_q^{\times}$ and an arbitrary integer $1\le r\le q-2$ , as an extension of Sun's determinant $S(d,p)$, we consider the following determinant:
$$
S_q(r,d):=\det\left[\chi^r(\alpha_i+d\alpha_j)\right]_{1\le i,j\le n}.
$$
We first obtain the following result:
\begin{theorem}\label{Thm. A}
Let $q=2n+1$ be an odd prime power and let $\chi$ be a generator of the cyclic group of all multiplicative characters of $\mathbb{F}_q$. For any $d\in\mathbb{F}_q^{\times}$ and $1\le r\le q-2$, the following results hold.

{\rm (i)} If $d\not\in\mathbb{F}_q^{\times2}$ and $r\equiv n\pmod2$, then $S_q(r,d)=0$.

{\rm (ii)} If $d=1$, then
$$
S_q(r,1)=\frac{(-1)^{\frac{n(n+1)}{2}-r}}{2^n}\cdot
	\begin{cases}
		\prod_{k=1}^n\left(J(\chi^{k-r},\chi^r)-J(\phi\chi^{k-r},\chi^r)\right)&\mbox{if}\ 4\mid q-3,\\
		\prod_{k=1}^n\left(J(\chi^{k-r},\chi^r)+J(\phi\chi^{k-r},\chi^r)\right)&\mbox{if}\ 4\mid q-1.
	\end{cases}
$$

{\rm (iii)} If $d\in\mathbb{F}_q^{\times2}$ and $q\equiv3\pmod4$, then
$S_q(r,d)=S_q(r,1)$. Furthermore, in this case,
\begin{equation}\label{Eq. a in Thm. A}
S_q(r,1)
=\frac{q^n}{2^n}\prod_{k=1}^{n}\left(_{2}F_1\left(
\begin{array}{cccccc}
 \varepsilon & \chi^r\\
  & \chi^k
\end{array}\bigg|1
\right)-\ _{2}F_1\left(
\begin{array}{cccccc}
 \phi & \chi^r\\
  & \chi^k
\end{array}\bigg|1
\right)\right).
\end{equation}

{\rm (iv)} If $d\in\mathbb{F}_q^{\times2}$ and $q\equiv1\pmod4$, then $S_q(r,d)=\delta_d S_q(r,1)$,
where
$$
\delta_d=\begin{cases}
 1   &\mbox{if}\ d\ \text{is a $4$th power in}\ \mathbb{F}_q,\\
-1 &\mbox{otherwise.}
\end{cases}
$$
Moreover, in this case,
\begin{equation}\label{Eq. b in Thm. A}
S_q(r,1)=\frac{q^n(-1)^{r+n(n+1)/2}}{2^n}\prod_{k=1}^n\ _{2}F_1\left(
\begin{array}{cccccc}
 \chi^{-r} & \chi^{2k-2r}\\
  & \chi^{2k-r}
\end{array}\bigg|-1
\right).
\end{equation}
\end{theorem}

\begin{remark}\label{Rem. A}
When $r=n$, the equation (\ref{Eq. b in Thm. A}) is indeed an analogue of certain product of gamma functions. In fact, by Kummer's theorem (cf. \cite[p. 9]{Bailey})
\begin{equation}\label{Eq. a in Rem. A}
_{2}F_1\left(
\begin{array}{cccccc}
 \alpha & \beta\\
  & 1-\alpha+\beta
\end{array}\bigg|-1\right)
=\frac{\Gamma(1+\beta-\alpha)\Gamma\left(1+\frac{\beta}{2}\right)}{\Gamma(1+\beta)\Gamma\left(1+\frac{\beta}{2}-\alpha\right)},
\end{equation}
where $\Gamma(z)$ is the gamma function. Also, by \cite[(4.12)]{Greene} we know that
$$
_{2}F_1\left(
\begin{array}{cccccc}
 \phi & \chi^{2k}\\
  & \phi\chi^{2k}
\end{array}\bigg|-1
\right)
$$
can be viewed as an analogue of
$$
_{2}F_1\left(
\begin{array}{cccccc}
 1/2 & k/n\\
  & 1-1/2+k/n
\end{array}\bigg|-1\right).
$$
Using Legendre's duplication formula
$$
\Gamma(2x)\sqrt{2\pi}=2^{2x-\frac{1}{2}}\Gamma(x)\Gamma\left(x+\frac{1}{2}\right),
$$
it is easy to verify that
\begin{equation*}
\prod_{k=1}^{n}\ _{2}F_1\left(
\begin{array}{cccccc}
 1/2 & k/n\\
  & 1-1/2+k/n
\end{array}\bigg|-1\right)
=\left(\frac{1}{2}\right)^{\frac{3n+1}{2}}\prod_{k=1}^n
\frac{\Gamma\left(\frac{2k}{n}\right)\Gamma\left(\frac{k}{2n}\right)^2}{\Gamma\left(\frac{k}{n}\right)^3}.
\end{equation*}
Hence (\ref{Eq. b in Thm. A}) can be viewed as an analogue of the above identity when $r=n$.
\end{remark}

In 2019, Sun \cite[Conjecure 4.5(iii)]{ffadeterminant} posed the following conjecture.
\begin{conjecture}\label{Conj. A}{\rm (Sun)} Let $p>3$ be a prime and let $d\in\mathbb{Z}$ with $p\nmid d$. Let
$$
D(d,p):=\det\bigg[(i^2+dj^2)\left(\frac{i^2+dj^2}{p}\right)\bigg]_{1\le i,j\le (p-1)/2},
$$
where $(\frac{\cdot}{p})$ is the Legendre symbol. Then
$$
\left(\frac{D(d,p)}{p}\right)=\begin{cases}
 (\frac{d}{p})^{(p-1)/4}   &\mbox{if}\ p \equiv 1 \pmod 4,\\
(\frac{d}{p})^{(p+1)/4}(-1)^{(h(-p)-1)/2} &\mbox{if} \ p\equiv 3 \pmod 4,
\end{cases}
$$
where $h(-p)$ is the class number of $\mathbb{Q}(\sqrt{-p})$.
\end{conjecture}
In this paper, we obtain the following generalized result.
\begin{theorem}\label{Thm. B}
Let $q=2n+1$ be an odd prime power and let $d\in\mathbb{F}_q^{\times}$. Suppose $\char(\mathbb{F}_q)> 3$. Let
$$
D(d,q)=\det\left[(\alpha_i+d\alpha_j)\phi(\alpha_i+d\alpha_j)\right]_{1\le i,j\le n}.
$$
If we view $D(d,q)$ as a determinant over $\mathbb{F}_q$, then the following results hold.

{\rm (i)} If $q\equiv1\pmod4$, then
$$
D(d,q)=d^{\frac{q-1}{4}}x_q(d)^2
$$
for some $x_q(d)\in\mathbb{F}_q$. In particular, if $q=p\equiv1\pmod4$ is a prime, then we further have $x_p(d)\in\mathbb{F}_p^{\times}$.

{\rm (ii)} If $q\equiv3\pmod4$, then
$$
D(d,q)=d^{\frac{q+1}{4}}(-1)^{\frac{q-3}{4}}\binom{(q+1)/2}{(q+1)/4}y_q(d)^2
$$
for some $y_q(d)\in\mathbb{F}_q$. In particular, if $q=p\equiv3\pmod4$ is a prime, then
$$
D(d,p)=d^{\frac{p+1}{4}}(-1)^{\frac{h(-p)-1}{2}}z_p(d)^2
$$
for some $z_p(d)\in\mathbb{F}_p^{\times}$, where $h(-p)$ is the class number of $\mathbb{Q}(\sqrt{-p})$.
\end{theorem}
As a direct consequence of Theorem \ref{Thm. B}, we confirm Sun's conjecture.
\begin{corollary}\label{Cor. B}
Conjecture \ref{Conj. A} holds.
\end{corollary}
The outline of this paper is as follows. The proof of Theorem \ref{Thm. A} will be given in Section 2. We will prove Theorem \ref{Thm. B} in Section 3.

\section{Proof of Theorem \ref{Thm. A}}
In this section, we assume that $q=2n+1$ is an odd prime power.

For every $d\in\mathbb{F}_q^{\times}$, clearly $\sigma_d: \alpha_i\mapsto d^2\alpha_i$ is a permutation of the set 
$\mathbb{F}_q^{\times2}=\{\alpha_1,\alpha_2,\cdots,\alpha_n\}$. We begin with the following lemma.
\begin{lemma}\label{Lem. A of Thm. A}
Let $\sgn(\sigma_d)$ denote the sign of $\sigma_d$. Then
$$
\sgn(\sigma_d)=\phi(d)^{n-1}.
$$
\end{lemma}
\begin{proof}
As $2\nmid q$, we can view $\sgn(\sigma_d$) as an element of $\mathbb{F}_q$. Then we have
$$
\sgn(\sigma_d)
=\prod_{1\le i<j\le n}\frac{d^2\alpha_j-d^2\alpha_i}{\alpha_j-\alpha_i}
=d^{n(n-1)}=\phi(d)^{n-1}.
$$
This completes the proof.
\end{proof}
Now we prove our first theorem.

{\bf Proof of Theorem \ref{Thm. A}.} (i) Suppose now that $d\in\mathbb{F}_q^{\times}\setminus\mathbb{F}_q^{\times2}$ and that $r\equiv n\pmod2$. Then by Lemma \ref{Lem. A of Thm. A} we have
\begin{align*}
S_q(r,d)&=\det\left[\chi^r(\alpha_j+d\alpha_i)\right]_{1\le i,j\le n}\\
&=\sgn(\sigma_d)\det\left[\chi^r(d^2\alpha_j+d\alpha_i)\right]_{1\le i,j\le n}
=(-1)^{n-1}\chi^r(d^n)S_q(r,d)\\
&=(-1)^{n-1+r}S_q(r,d).
\end{align*}
The last equality follows from $\chi^n=\phi$ and $\chi^r(d^n)=\chi^n(d^r)=\phi(d)^r=(-1)^r$. 
As $r\equiv n\pmod2$, we have $S_q(r,d)=-S_q(r,d)$ and hence $S_q(r,d)=0$.

(ii)-(iv) Suppose that $d=d_0^2$ for some $d_0\in\mathbb{F}_q^{\times}$.
Then by Lemma \ref{Lem. A of Thm. A}
$$
S_q(r,d)=S_q(r,d_0^2)=\sgn(\sigma_{d_0})S_q(r,1)=\phi(d_0)^{n-1}S_q(r,1).
$$
Hence we obtain that $S_q(r,d)=S_q(r,1)$ if $q\equiv3\pmod4$ and that
$S_q(r,d)=\delta_d S_q(r,1)$ if $q\equiv1\pmod4$,
where
$$
\delta_d=\begin{cases}
 1   &\mbox{if}\ d\ \text{is a $4$th power in}\ \mathbb{F}_q,\\
-1 &\mbox{otherwise.}
\end{cases}
$$
Now we focus on $S_q(r,1)$. Since
$$
T^n-1=\prod_{j=1}^n\left(T-\alpha_j\right),
$$
we obtain
$$
\prod_{j=1}^n\alpha_j=(-1)^{n-1}.
$$
Combining the above with the easily-proved fact that $\chi(-1)=-1$, we obtain that
\begin{equation}\label{Eq. det of Aq(r)}
S_q(r,1)=(-1)^{r(n-1)}\det A_q(r),
\end{equation}
where
$A_q(r)=[\chi^r(\alpha_i+\alpha_j)\chi^{-r}(\alpha_j)]_{1\le i,j\le n}$. For $k=1,2,\cdots,n$, let
$$
\lambda_k:=\sum_{j=1}^n\chi^r(1+\alpha_j)\chi^{k-r}(\alpha_j)
$$
and
$$
{\bf v}_k:=\left(\chi^k(\alpha_1),\chi^k(\alpha_2),\cdots,\chi^k(\alpha_n)\right)^T,
$$
where $M^T$ denotes the transpose of a matrix $M$. For each $i$, we have 
\begin{align*}
\sum_{j=1}^n\chi^r(\alpha_i+\alpha_j)\chi^{-r}(\alpha_j)\chi^k(\alpha_j)
&=\sum_{j=1}^n\chi^r\left(1+\frac{\alpha_j}{\alpha_i}\right)\chi^{k-r}\left(\frac{\alpha_j}{\alpha_i}\right)
\chi^k(\alpha_i)\\
&=\sum_{j=1}^n\chi^r(1+\alpha_j)\chi^{k-r}(\alpha_j)\chi^k(\alpha_i)\\
&=\lambda_k\chi^k(\alpha_i).
\end{align*}
This implies that
$$
A_q(r){\bf v}_k=\lambda_k{\bf v}_k
$$
for $k=1,2,\cdots,n$.
As
$$
\left| \begin{array}{cccccccc}
\chi^1(\alpha_1) & \chi^2(\alpha_1) & \ldots  & \chi^n(\alpha_1)  \\
\chi^1(\alpha_2) & \chi^2(\alpha_2) &  \ldots  &  \chi^n(\alpha_2)\\
\vdots & \vdots & \ddots  & \vdots    \\
\chi^1(\alpha_n)& \chi^2(\alpha_n) & \ldots &  \chi^n(\alpha_n)\\
\end{array} \right|^2=\prod_{1\le i<j\le n}\(\chi(\alpha_j)-\chi(\alpha_i)\)^2\ne0,
$$
the vectors ${\bf v}_1,\cdots,{\bf v}_n$ are linearly independent. Hence $\lambda_1,\cdots,\lambda_n$ are exactly all eigenvalues of $A_q(r)$ (counting multiplicities). By the above we obtain
\begin{equation}\label{Eq. product of eigenvalues}
S_q(r,1)=(-1)^{r(n-1)}\det A_q(r)=(-1)^{r(n-1)}\prod_{k=1}^n\lambda_k.
\end{equation}
Note that 
$$
\frac{\varepsilon(x)+\phi(x)}{2}=\begin{cases}1&\mbox{if}\ x\in\mathbb{F}_q^{\times2},
	\\0&\mbox{otherwise.}\end{cases}
$$
Hence 
\begin{equation}\label{Eq. eigenvalues}
\lambda_k=\frac{1}{2}\sum_{x\in\mathbb{F}_q}(\varepsilon(x)+\phi(x))\chi^r(1+x)\chi^{k-r}(x).
\end{equation}
Observe that 
\begin{align*}
	\frac{1}{2}\sum_{x\in\mathbb{F}_q}\varepsilon(x)\chi^r(1+x)\chi^{k-r}(x)
	&=\frac{(-1)^{k-r}}{2}\sum_{x\in\mathbb{F}_q}\varepsilon(x)\chi^r(1-x)\chi^{k-r}(x)\\
	&=\frac{(-1)^{k-r}}{2}J(\chi^{k-r},\chi^r),
\end{align*}
and that 
\begin{align*}
	\frac{1}{2}\sum_{x\in\mathbb{F}_q}\phi(x)\chi^r(1+x)\chi^{k-r}(x)
	&=\frac{(-1)^{k-r+(q-1)/2}}{2}\sum_{x\in\mathbb{F}_q}\phi(x)\chi^r(1-x)\chi^{k-r}(x)\\
	&=\frac{(-1)^{k-r+(q-1)/2}}{2}J(\phi\chi^{k-r},\chi^r).
\end{align*}

Hence, (\ref{Eq. eigenvalues}) becomes
\begin{equation}\label{Eq. jacobi sums and eigenvalues}
\lambda_k=\frac{(-1)^{k-r}}{2}
\begin{cases}
	\left(J(\chi^{k-r},\chi^r)-J(\phi\chi^{k-r},\chi^r)\right)&\mbox{if}\ q\equiv3\pmod4,\\
	\left(J(\chi^{k-r},\chi^r)+J(\phi\chi^{k-r},\chi^r)\right)&\mbox{if}\ q\equiv1\pmod4.
\end{cases}
\end{equation}
It is well-known that the above Jacobi sums can be expressed in terms of finite field hypergeometric functions. 

In fact, in the case $q\equiv3\pmod4$, by (\ref{Eq. 2F1}) we obtain that 
\begin{equation}\label{Eq. 2F1 for trivial character}
\begin{split}
\frac{1}{2}\sum_{x\in\mathbb{F}_q}\varepsilon(x)\chi^r(1+x)\chi^{k-r}(x) 
=&\frac{(-1)^{k-r}}{2}\sum_{x\in\mathbb{F}_{q}}\varepsilon(x)\chi^r(1-x)\chi^{k-r}(x)\\
=&\frac{(-1)^{k-r}}{2}\sum_{x\in\mathbb{F}_q}\varepsilon(1-x)\chi^r(x)\chi^{k-r}(1-x)\\
=&\frac{q}{2}\
_{2}F_1\left(
\begin{array}{cccccc}
 \varepsilon & \chi^r\\
  & \chi^k
\end{array}\bigg|1
\right),
\end{split}
\end{equation}
and that 
\begin{equation}\label{Eq. 2F1 for phi}
\begin{split}
\frac{1}{2}\sum_{x\in\mathbb{F}_q}\phi(x)\chi^r(1+x)\chi^{k-r}(x)
=&\frac{(-1)^{k-r+1}}{2}\sum_{x\in\mathbb{F}_q}\phi(x)\chi^r(1-x)\chi^{k-r}(x)\\
=&\frac{(-1)^{k-r+1}}{2}\sum_{x\in\mathbb{F}_q}\phi(1-x)\chi^r(x)\chi^{k-r}(1-x)\\
=&\frac{-q}{2}\
_{2}F_1\left(
\begin{array}{cccccc}
 \phi & \chi^r\\
  & \chi^k
\end{array}\bigg|1
\right).
\end{split}
\end{equation}
Combining (\ref{Eq. 2F1 for trivial character}) and (\ref{Eq. 2F1 for phi}) with (\ref{Eq. eigenvalues}), we obtain 
\begin{align*}
\lambda_k
=\frac{q}{2}\left(_{2}F_1\left(
\begin{array}{cccccc}
 \varepsilon & \chi^r\\
  & \chi^k
\end{array}\bigg|1
\right)-{}_{2}F_1\left(
\begin{array}{cccccc}
 \phi & \chi^r\\
  & \chi^k
\end{array}\bigg|1
\right)\right).
\end{align*}
This, together with (\ref{Eq. product of eigenvalues}), implies the desired result in Theorem \ref{Thm. A}(iii).

In the case $q\equiv1\pmod4$, we first have 
\begin{align*}
\lambda_k=\frac{1}{2}\sum_{x\in\mathbb{F}_q}\chi^r(1+x^2)\chi^{2k-2r}(x)
&=\frac{1}{2}\sum_{x\in\mathbb{F}_q}\chi^r(1-x^2)\chi^{2k-2r}(\sqrt{-1}x)\\
&=\frac{(-1)^{k-r}}{2}\sum_{x\in\mathbb{F}_q}\chi^r(1-x^2)\chi^{2k-2r}(x),
\end{align*}
where $\sqrt{-1}\in\mathbb{F}_q$ such that $(\sqrt{-1})^2=-1$. By (\ref{Eq. 2F1}), we obtain 
\begin{equation}\label{Eq. 2F1 for the case q=1 mod 4}
\lambda_k=\frac{(-1)^kq}{2}\
_{2}F_1\left(
\begin{array}{cccccc}
 \chi^{-r} & \chi^{2k-2r}\\
  & \chi^{2k-r}
\end{array}\bigg|-1
\right).
\end{equation}
Combining this with (\ref{Eq. product of eigenvalues}), we obtain
$$
S_q(r,1)=\frac{q^n(-1)^{r+n(n+1)/2}}{2^n}\prod_{k=1}^n\ _{2}F_1\left(
\begin{array}{cccccc}
 \chi^{-r} & \chi^{2k-2r}\\
  & \chi^{2k-r}
\end{array}\bigg|-1
\right).
$$

In view of the above, the proof of Theorem \ref{Thm. A} is now complete. \qed

\section{Proof of Theorem \ref{Thm. B}}
In this section, we assume that $q=2n+1$ is an odd prime power with $3\nmid q$. Also, for any $x,y\in\mathbb{F}_q$, we write $x\equiv y\pmod{\mathbb{F}_q^{\times2}}$ if there is an element $z\in\mathbb{F}_q^{\times}$ such that $x=yz^2$.

Let $a$ be an integer with $(a,n)=1$. We know that multiplication by $a$ induces
a permutation $\tau_a$ of $\Z/n\Z=\{\overline{0},\overline{1},...,\overline{n-1}\}$. Lerch \cite{ML} obtained the following result which determines the sign of $\tau_a$.
\begin{lemma}\label{Lem. A of Thm. B}
Let $\sgn(\tau_a)$ denote the sign of $\tau_a$. Then $$\sgn(\tau_a)=\begin{cases}(\frac{a}{n})&\mbox{if $n$ is odd},\\1&\mbox{if}\
n\equiv2\pmod4,\\(-1)^{\frac{a-1}{2}}&\mbox{if}\ n\equiv0\pmod4,\end{cases}$$
where $(\frac{\cdot}{n})$ denotes Jacobi symbol if $n$ is odd.
\end{lemma}
Clearly $\inv_q: \alpha_i\mapsto\alpha_i^{-1}$ is a permutation of $\alpha_1,\alpha_2,\cdots,\alpha_n$. Fix a generator $g$ of the cyclic group $\mathbb{F}_q^{\times}$. 
If we reindex the elements $g^0,g^2,\cdots,g^{2(n-1)}$ of $\mathbb{F}_q^{\times2}$ as $\overline{0},\overline{1},\cdots,\overline{n-1}$ of $\mathbb{Z}/n\mathbb{Z}$, then the permutation $\inv_q$ becomes the permutation $\tau_{-1}$ that sends $\overline{i}$ to $\overline{n-i}$ for $1\le i\le n-1$ while leaving $\overline{0}$ fixed. This permutation has $\frac{(n-1)(n-2)}{2}$ inversions (indeed, its inversions are all pairs $(\overline{i},\overline{j})$ with $0<i<j<n$, and there are $(n-1)(n-2)/2$ such pairs).
In view of the above we obtain the following result.
\begin{lemma}\label{Lem. B of Thm. B}
Let notations be as above. Then
$$
\sgn(\inv_q)=\sgn(\tau_{-1})=(-1)^{\frac{(n-1)(n-2)}{2}}=\begin{cases}(-1)^{\frac{n-1}{2}}&\mbox{if $q\equiv3\pmod4$},\\1&\mbox{if}\
q\equiv5\pmod8,\\-1&\mbox{if}\ q\equiv1\pmod8.\end{cases}
$$
\end{lemma}
We also need the following lemma (cf. \cite[Lemma 10]{K2}).
\begin{lemma}\label{Lem. C of Thm. B}
Let $R$ be a commutative ring and let $m$ be a positive integer. Set $P(T)=p_{m-1}T^{m-1}+\cdots+p_1T+p_0\in R[T]$. Then
$$
\det[P(X_iY_j)]_{1\le i,j\le m}
=\prod_{i=0}^{m-1}p_i\prod_{1\le i<j\le m}\left(X_i-X_j\right)\left(Y_i-Y_j\right).
$$
\end{lemma}
We are now in a position to prove our main result.

{\bf Proof of Theorem \ref{Thm. B}.} Recall that
$$
D(d,q)=\det\left[(\alpha_i+d\alpha_j)\phi(\alpha_i+d\alpha_j)\right]_{1\le i,j\le n}.
$$
If we view $D(d,q)$ as a determinant over $\mathbb{F}_q$, then
$$
D(d,q)=\det\left[\left(\alpha_i+d\alpha_j\right)^{\frac{q+1}{2}}\right]_{1\le i,j\le n}.
$$
Thus
\begin{equation}\label{Eq. 3F1}
D(d,q)
\equiv\det\left[\left(\frac{\alpha_i}{\alpha_j}+d\right)^{\frac{q+1}{2}}\right]_{1\le i,j\le n}\pmod{\mathbb{F}_q^{\times2}}.
\end{equation}
Noting that
$$
\left(\frac{\alpha_i}{\alpha_j}\right)^{\frac{q-1}{2}}=1\ \text{and}\ d^{\frac{q-1}{2}}=\phi(d), 
$$
one can verify that
\begin{align*}
&\left(\frac{\alpha_i}{\alpha_j}+d\right)^{\frac{q+1}{2}}\\
=&\sum_{k=0}^{\frac{q+1}{2}}\binom{\frac{q+1}{2}}{k}\left(\frac{\alpha_i}{\alpha_j}\right)^kd^{\frac{q+1}{2}-k}\\
=&\sum_{k=2}^{\frac{q-3}{2}}\binom{\frac{q+1}{2}}{k}\left(\frac{\alpha_i}{\alpha_j}\right)^kd^{\frac{q+1}{2}-k}+\left(\frac{1}{2}+\phi(d)\right)d+\left(\frac{\phi(d)}{2}+1\right)\frac{\alpha_i}{\alpha_j}.
\end{align*}
Hence 
$$\left(\frac{\alpha_i}{\alpha_j}+d\right)^{\frac{q+1}{2}}=f\left(\frac{\alpha_i}{\alpha_j}\right),$$
where
$$
f(T)
=\left(\frac{1}{2}+\phi(d)\right)d+\left(\frac{\phi(d)}{2}+1\right)T
+\sum_{k=2}^{\frac{q-3}{2}}\binom{\frac{q+1}{2}}{k}d^{\frac{q+1}{2}-k}T^k.
$$
Observe that 
$$
\left(\frac{1}{2}+\phi(d)\right)\left(\frac{\phi(d)}{2}+1\right)\equiv \phi(d)\pmod{\mathbb{F}_q^{\times2}}
$$
(just check both cases $\phi(d)=1$ and $\phi(d)=-1$). By this observation, if we let $C_f$ denote the product of the coefficients of $f(T)$, then 
\begin{align*}
C_f
&=d\left(\frac{1}{2}+\phi(d)\right)\left(\frac{\phi(d)}{2}+1\right)\times\prod_{k=2}^{\frac{q-3}{2}}\binom{\frac{q+1}{2}}{k}d^{\frac{q+1}{2}-k}\\
&=d\frac{\left(\frac{1}{2}+\phi(d)\right)\left(\frac{\phi(d)}{2}+1\right)}
{\left(\frac{q+1}{2}\cdot d^{(q+1)/2}\right)^2}
\times\prod_{k=0}^{\frac{q+1}{2}}\binom{\frac{q+1}{2}}{k}d^{\frac{q+1}{2}-k}\\
&\equiv d\phi(d)\times\prod_{k=0}^{\frac{q+1}{2}}\binom{\frac{q+1}{2}}{k}d^{\frac{q+1}{2}-k}\pmod{\mathbb{F}_q^{\times2}}
\end{align*} 
Thus, applying Lemma \ref{Lem. C of Thm. B} to $P(T)=f(T)$, we rewrite (\ref{Eq. 3F1}) as
\begin{align*}
D(d,q)
&\equiv\sgn(\inv_q)\cdot d\phi(d)\cdot\prod_{1\le i<j\le n}\left(\alpha_j-\alpha_i\right)^2\cdot
\prod_{k=0}^{\frac{q+1}{2}}\binom{\frac{q+1}{2}}{k}d^{\frac{q+1}{2}-k}\\
&\equiv\sgn(\inv_q)\cdot d\phi(d)\cdot\prod_{k=0}^{\frac{q+1}{2}}\binom{\frac{q+1}{2}}{k}d^{\frac{q+1}{2}-k}\\
&\equiv\sgn(\inv_q)\cdot d^{1+(n+1)(n+2)/2}\phi(d)\cdot\prod_{k=0}^{\frac{q+1}{2}}\binom{\frac{q+1}{2}}{k}
\pmod{\mathbb{F}_q^{\times2}}.
\end{align*}
Now we divide the remaining proof into two cases.

{\bf Case 1.} $q\equiv1\pmod4$.

In this case, since $\pm1\in\mathbb{F}_q^{\times2}$ and $\char(\mathbb{F}_q)>3$, one can verify that
$$
\sgn(\inv_q)\cdot d^{1+(n+1)(n+2)/2}\phi(d)
\equiv d^{\frac{q-1}{4}}\pmod{\mathbb{F}_q^{\times2}},
$$
and that
$$
\prod_{k=0}^{\frac{q+1}{2}}\binom{\frac{q+1}{2}}{k}\in\{0\}\cup\mathbb{F}_q^{\times2}.
$$
Hence there is an element $x_q(d)\in\mathbb{F}_q$ such that
$$
D(d,q)=d^{\frac{q-1}{4}}x_q(d)^2.
$$
Moreover, if $q=p\equiv1\pmod4$ is a prime, then clearly
$$
\prod_{k=0}^{\frac{p+1}{2}}\binom{\frac{p+1}{2}}{k}\in\mathbb{F}_p^{\times2}.
$$
Hence we can further obtain that $x_p(d)\in\mathbb{F}_p^{\times}$.

{\bf Case 2.} $q\equiv3\pmod4$.

In this case, noting that $d\equiv\phi(d)\pmod{\mathbb{F}_q^{\times2}}$, by Lemma \ref{Lem. B of Thm. B} we obtain that
$$
\sgn(\inv_q)\cdot d^{1+(n+1)(n+2)/2}\phi(d)
\equiv (-1)^{\frac{q-3}{4}}d^{\frac{q+1}{4}}\pmod{\mathbb{F}_q^{\times2}},
$$
and that
$$
\prod_{k=0}^{\frac{q+1}{2}}\binom{\frac{q+1}{2}}{k}
\in\{0\}\cup\binom{(q+1)/2}{(q+1)/4}\mathbb{F}_q^{\times2}.
$$
In view of the above there is an element $y_q(d)$ such that
$$
D(d,q)=d^{\frac{q+1}{4}}(-1)^{\frac{q-3}{4}}\binom{(q+1)/2}{(q+1)/4}y_q(d)^2.
$$
Moreover, if $q=p\equiv3\pmod4$ is a prime, then clearly
$$
\prod_{k=0}^{\frac{p+1}{2}}\binom{\frac{p+1}{2}}{k}
\in\binom{(p+1)/2}{(p+1)/4}\mathbb{F}_p^{\times2}.
$$
Note that
$$
\binom{(p+1)/2}{(p+1)/4}\equiv \frac{1}{2}\cdot\frac{p-1}{2}!
\equiv (-1)^{\frac{h(-p)+1}{2}+\frac{p+1}{4}}\pmod{\mathbb{F}_p^{\times2}}.
$$
The last congruence follows from $2\equiv(-1)^{\frac{p+1}{4}}\pmod{\mathbb{F}_p^{\times2}}$ and Mordell's result \cite{Mordell} which states that if $p\equiv3\pmod4$ is a prime greater that $3$, then
$$
\frac{p-1}{2}!\equiv(-1)^{\frac{h(-p)+1}{2}}\pmod p.
$$
By the above results, we see that there is an element $z_p(d)\in\mathbb{F}_p^{\times}$ such that
$$
D(d,p)=d^{\frac{p+1}{4}}(-1)^{\frac{h(-p)-1}{2}}z_p(d)^2.
$$

This completes the proof.\qed

\Ack\ We would like to thank the two anonymous referees for their helpful comments. The first author also thanks Prof. Hao Pan and Dr. He-Xia Ni for their steadfast encouragement.

This work was supported by the National Natural Science Foundation of China (Grant No. 12101321 and Grant No. 11971222) and the Natural Science Foundation of the Higher Education Institutions of Jiangsu Province (Grant No. 21KJB110002 and Grant No. 21KJB110001).


\begin{thebibliography}{99}
\baselineskip=17pt
\bibitem{Bailey} W. N. Bailey, Generalized hypergeometric series, Cambridge Univ. Press, Cambridge, 1935.
\bibitem{Bar1} R. Barman, G. Kalita, Hypergeometric functions over $\mathbb{F}_q$ and traces of Frobenius for elliptic curves, Proc. Amer. Math. Soc. 141 (2013), 3403--3410.
\bibitem{carlitz}  L. Carlitz, Some cyclotomic matrices, Acta Arith. 5 (1959), 293--308.
\bibitem{chapman} R. Chapman, Determinants of Legendre symbol matrices, Acta Arith. 115 (2004), 231--244.
\bibitem{evil} R. Chapman, My evil determinant problem, preprint, December 12, 2012, available from http://empslocal.ex.ac.uk/people/staff/rjchapma/etc/evildet.pdf.
\bibitem{Fuselier} J. G. Fuselier, Hypergeometric functions over $\mathbb{F}_p$ and relations to elliptic curves and modular forms. Proc. Amer. Math. Soc. 138 (2010), 109--123.
\bibitem{Greene} J. Greene, Hypergeometric functions over finite fields, Trans. Amer. Math. Soc. 301 (1987), 77--101.
\bibitem{Dimitry} D. Krachun, F. Petrov, Z.-W. Sun, M. Vsemirnov, On some determinants involving Jacobi symbols, Finite Fields Appl. 64 (2020), Article 101672.
\bibitem{K2}  C. Krattenthaler, Advanced determinant calculus: a complement, Linear Algebra
Appl. 411 (2005), 68--166.
\bibitem{Lehmer} D. H. Lehmer, On certain character matrices, Pacific J. Math. 6 (1956), 491--499.
\bibitem{ML} M. Lerch, Sur un th\'{e}or\`{e}me arithm\'{e}tique de Zolotarev, Bull. Intern. de l'Acad. Fran\c{c}ois Joseph 3 (1896), 34--37.
\bibitem{Mordell} L. J. Mordell, The congruence $((p-1)/2)!\equiv\pm1\pmod p$, Am. Math. Mon. 68 (1961), 145--146.
\bibitem{Ono} K. Ono, Values of gaussian hypergeometric series, Trans. Amer. Math. Soc. 350 (1998), 1205--1223.
\bibitem{ffadeterminant} Z.-W. Sun, On some determinants with Legendre symbol entries, Finite Fields Appl. 56 (2019), 285--307.
\bibitem{Wu} H.-L. Wu, Elliptic curves over $\mathbb{F}_p$ and determinants of Legendre matrices, Finite Fields Appl. 76 (2021), Article 101929.
\bibitem{M1} M. Vsemirnov, On the evaluation of R. Chapman's ``evil determinant", Linear Algebra Appl. 436 (2012), 4101--4106.
\bibitem{M2} M. Vsemirnov, On R. Chapman's ``evil determinant": case $ p \equiv 1 \pmod 4$, Acta Arith. 159 (2013), 331--344.
\end{thebibliography}
\end{document}